\documentclass[11pt]{article}
\usepackage{amsthm,geometry,amssymb,amsmath,enumerate,float,tikz, cite}
\tikzstyle{every node}=[circle, draw, inner sep=0pt, minimum width=4pt]
\RequirePackage[utf8]{inputenc}				
\newtheorem{theorem}{Theorem}
\newtheorem{lemma}{Lemma}[section]
\newtheorem{claim}{Claim}

\title{On the hardness of deciding the equality of the induced and the uniquely restricted matching number}
\author{M. Fürst}
\date{}

\begin{document}

 \maketitle
 \begin{center}
 {\small 
Institute of Optimization and Operations Research, Ulm University, Germany\\
\texttt{maximilian.fuerst@uni-ulm.de}\\[3mm]
}
\end{center}

\begin{abstract}
If $G(M)$ denotes the subgraph of a graph $G$ 
induced by the set of vertices that are covered by some matching $M$ in $G$, 
then $M$ is 
an induced 
or a uniquely restricted matching 
if $G(M)$ is $1$-regular 
or if $M$ is the unique perfect matching of $G(M)$, respectively.
Let $\nu_s(G)$ and $\nu_{ur}(G)$ denote
the maximum cardinality of 
an induced and 
a uniquely restricted matching in $G$.
Golumbic, Hirst, and Lewenstein 
(Uniquely restricted matchings, Algorithmica 31 (2001) 139-154)
posed the problem to characterize the graphs $G$ with $\nu_{ur}(G) = \nu_{s}(G)$.
We prove that the corresponding decision problem is NP-hard, which suggests
that a good characterization is unlikely to be possible.
\end{abstract}
 
 {\small 
\begin{tabular}{lp{13cm}}
\end{tabular}
}

\section{Introduction}
We consider only simple, finite, and undirected graphs, and use standard terminology.
For a graph $G$, and a matching $M$ in $G$, let $V(M)$ be the set of vertices that are
covered by $M$, and let $G(M)$ be the subgraph of $G$ induced by $V(M)$. 
A matching $M$ in $G$ is 
\begin{itemize}
 \item \textit{induced} \cite{ca1} if $G(M)$ is $1$-regular,
 \item \textit{acyclic} \cite{gohehela} if $G(M)$ is a forest, or
 \item \textit{uniquely restricted} \cite{gohile} if $M$ is the unique perfect matching of $G(M)$.
\end{itemize}
The maximum cardinality of an ordinary, a uniquely restricted, an acyclic, and an induced matching is denoted by
$\nu(G)$, $\nu_{ur}(G)$, $\nu_{ac}(G)$, and $\nu_s(G)$, respectively.
While the ordinary matching number is tractable \cite{edmonds},
the three remaining restricted matching numbers are NP-hard \cite{gohehela, gohile, stva}.
Golumbic et al. \cite{gohile} observed that a matching $M$ 
in $G$ is uniquely restricted if and only if $G$ contains no $M$-alternating cycle, which implies that
$\nu_{ur}(G) \geq \nu_{ac}(G)$. 
Since every induced matching is also acyclic,
we obtain that
\begin{align} \label{key_ineq}
 \nu(G) \geq \nu_{ur}(G) \geq \nu_{ac}(G) \geq \nu_s(G).
\end{align}
In order to understand how those different restricted matching numbers relate
to each other, it seems to be interesting to characterize the graphs
achieving equality in one or more of the inequalities in (\ref{key_ineq}). 
On the positive side, deciding whether a given graph $G$ satisfies $\nu(G) = \nu_s(G)$ or $\nu(G) = \nu_{ur}(G)$,
and deciding whether a given subcubic graph $G$ satisfies $\nu_{ur}(G) = \nu_s(G)$
is tractable
\cite{cawa, jora, koro,lema, peraso, fura2}. 
On the negative side, 
the hardness of deciding $\nu(G) = \nu_{ac}(G)$ and $\nu_{ur}(G) = \nu_{ac}(G)$ was shown in \cite{fura}.
In 2001, Golumbic et al. \cite{gohile} posed the problem to characterize the graphs $G$ with $\nu_{ur}(G) = \nu_{s}(G)$.
In this short note, we will prove the hardness of deciding 
$\nu_{ur}(G) = \nu_{ac}(G)$ and $\nu_{ur}(G) = \nu_s(G)$.
This shows that a good characterization does not exist unless NP $=$ co-NP.
Note that it is not obvious whether these two decision problems belong to NP.
\begin{theorem} \label{t1}
 Deciding whether a given graph $G$ of maximum degree $4$ satisfies $\nu_{ac}(G) = \nu_s(G)$ is NP-hard.
\end{theorem}
\begin{theorem} \label{t2}
Deciding whether a given bipartite graph $G$ satisfies $\nu_{ur}(G) = \nu_s(G)$ is NP-hard. 
\end{theorem}
The proofs of Theorem \ref{t1} and \ref{t2} are postponed to the following sections.
We close the introduction with a few notations. 
For a graph $G$ and
two disjoint sets $X$ and $Y$ of vertices of $G$, let 
$$E_G(X, Y) = \{uv \in E(G) : u \in X, v \in Y\},$$
and let $E_G(X) = E(G[X])$.
For every positive integer $k$, let $[k] = \{1,\ldots,k\}$.
\section{Proof of Theorem \ref{t1}}
 We prove the statement by a reduction from 
 {\sc Satisfiability} that remains NP-complete (cf. e.g. \cite{gajo}) for instances
 where
 every clause contains two or three literals,
 every positive literal appears in at most two different clauses,
 every negative literal appears in at most one clause,
 and no clause contains a literal and its negation.
 Let $\Gamma$ be such an instance of {\sc Satisfiability}
 with variables $x_1,\ldots,x_n$ and
 clauses $c_1,\ldots,c_m$. 
 For every $j \in [m]$, 
 let $|c_j|$ be the number of literals that belong to the clause $c_j$.
 
 Let $G$ be a graph that arises from the union of 
 $n$ disjoint triangles with vertex sets $X_1,\ldots,X_n$,
 and $m$ disjoint cliques with vertex sets $C_1,\ldots,C_m$, 
 where $|C_j| = |c_j| + 1$ for every $j \in \lbrack m \rbrack$.
 For every $i \in \lbrack n \rbrack$,
 let $X_i = \{t_i,f_i,u_i\}$.
 For every $j \in \lbrack m \rbrack$, 
 identify $|c_j|$ vertices in $C_j$ with the literals in the
 clause $c_j$, and let $v_j$ be the vertex in $C_j$ that is not identified with a literal from the clause $c_j$.
 Let $i$ be in $[n]$,
 let $w_1$ and $w_2$ be the vertices in $\bigcup_{j=1}^m C_j$
 that are identified with the literal $x_i$,
 and let $w_3$ be the vertex in $\bigcup_{j=1}^m C_j$
 that is identified with the literal $\bar{x}_i$.
 Now, add the edges $f_iw_1$, $f_iw_2$, and $t_iw_3$, see Figure \ref{fig1}
 for an illustration.
 \begin{figure}[H]
\centering\tiny
\begin{tikzpicture}[scale = 0.9] 
	    \foreach \i in {1,2,3} {
	    \node (u\i) at (5*\i,3) {};
	    \node (f\i) at (-0.5+5*\i,2) {};
	    \node (t\i) at (0.5+5*\i,2) {};
	    \draw[-] (u\i) -- (t\i) -- (f\i) -- (u\i);
	    }
	    \foreach \i in {1,3} {
	    \draw[-,dotted] (f\i) -- (4.25 + 5*\i -5,1.5);
	    \draw[-,dotted] (f\i) -- (4.75 + 5*\i -5,1.5);
	    \draw[-,dotted] (t\i) -- (5.75 + 5*\i -5,1.5);
	    }
	    \pgftext[x=5cm,y=3.25cm]  {\footnotesize $u_1$};
	    \pgftext[x=10cm,y=3.25cm] {\footnotesize $u_{i}$};
	    \pgftext[x=15cm,y=3.25cm] {\footnotesize $u_{n}$};
	    \pgftext[x=3.5cm,y=-0.25cm] {\footnotesize $v_{1}$};
	    \pgftext[x=8.5cm,y=-0.25cm] {\footnotesize $v_{j}$};
	    \pgftext[x=10.5cm,y=-0.25cm] {\footnotesize $v_{k}$};
	    \pgftext[x=12.5cm,y=-0.25cm] {\footnotesize $v_{l}$};
	    \pgftext[x=17.5cm,y=-0.25cm] {\footnotesize $v_{m}$};
	    \pgftext[x=9.2cm,y=2cm] {\footnotesize $f_{i}$};
	    \pgftext[x=10.8cm,y=2cm] {\footnotesize $t_{i}$};
	    \pgftext[x=7.7cm,y=0.75cm] {\footnotesize $w_{1}$};
	    \pgftext[x=9.7cm,y=0.75cm] {\footnotesize $w_{2}$};
	    \pgftext[x=11.7cm,y=0.75cm] {\footnotesize $w_{3}$};
	    \draw[-,dotted] (7,2.5) -- (8,2.5);
	    \draw[-,dotted] (12,2.5) -- (13,2.5);
	    
	    \foreach \i in {1,2,3} {
	    \node (a\i) at (6+2*\i,0.75) {};
	    \node (b\i) at (6.5+2*\i,0.75) {};
	    \node (c\i) at (7+2*\i,0.75) {};
	    \node (d\i) at (6.5+2*\i,0) {};
	    \draw[-] (a\i) -- (b\i) -- (c\i) -- (d\i) -- (a\i);
	    \draw[-] (d\i) -- (b\i);
	    \draw[-] (a\i) to[out=30,in=150] (c\i);
	    }	    
	    \draw[-] (f2) -- (a1);
	    \draw[-] (f2) -- (a2);
	    \draw[-] (t2) -- (a3);
	    \draw[-,dotted] (14.5,0.375) -- (15.5,0.375);
	    \draw[-,dotted] (6.5,0.375) -- (5.5,0.375);
	    
	    \foreach \i in {4,5} {
	    \node (a\i) at (-53+14*\i,0.75) {};
	    \node (b\i) at (-52.5 +14*\i,0.75) {};
	    \node (c\i) at (-52 +14*\i,0.75) {};
	    \node (d\i) at (-52.5 +14*\i,0) {};
	    \draw[-] (a\i) -- (b\i) -- (c\i) -- (d\i) -- (a\i);
	    \draw[-] (d\i) -- (b\i);
	    \draw[-] (a\i) to[out=30,in=150] (c\i);
	    }
\end{tikzpicture}
\caption{The construction for the proof of Theorem \ref{t1}.} \label{fig1}
\end{figure}
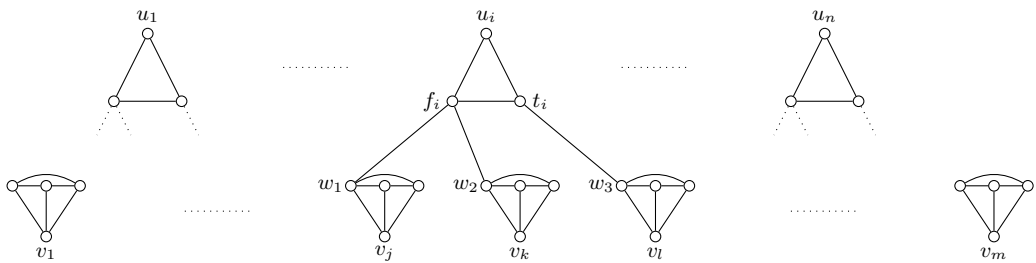
 
 \begin{lemma} \label{l_acy}
  $\nu_{ac}(G) = n+m$.
 \end{lemma}
\begin{proof}
 Let $M$ arise from $\{u_i t_i : i \in \lbrack n \rbrack \}$ by 
 adding, for every $j \in \lbrack m \rbrack$,
 an arbitrary edge of $E_G(C_j)$.
 Suppose, for a contradiction, that $M$ is not an acyclic matching in $G$.
 Since $d_{G(M)}(u_i) = 1$ and $d_{G(M)}(t_i) \leq 2$,
 neither $u_i$ nor $t_i$ are contained in any cycle of $G(M)$ 
 for every $i \in \lbrack n \rbrack$.
 By construction, this implies a contradiction.
 Hence, $\nu_{ac}(G) \geq n+m$.
 
 Let $M = M_1 \cup M_2 \cup M_3$ be some maximum acyclic matching with
 \begin{itemize}
  \item $M_1 \subseteq \bigcup_{i=1}^n E_G(X_i)$,
  \item $M_2 \subseteq \bigcup_{j=1}^m E_G(C_j)$, and
  \item $M_3 \subseteq E_G \big( \bigcup_{i=1}^n{X_i}, V(G) \setminus \bigcup_{i=1}^n{X_i} \big)$
 \end{itemize}
 minimizing $|M_3|$.
 Let $j \in \lbrack m \rbrack$.
 If $E_G(C_j, V(G) \setminus C_j) \cap M_3 \neq \emptyset$, then $E_G(C_j) \cap M_2 = \emptyset$. 
 Since $M$ is acyclic, we have that $|E_G(C_j, V(G) \setminus C_j) \cap M_3| \leq 2$, and,
 if $E_G(C_j, V(G) \setminus C_j)\cap M_3$ contains exactly one edge $uv$ with $u \in C_j$, then 
 the matching $(M \cup \{v_ju\}) \setminus \{uv\}$ is acyclic, 
 which is a contradiction to the minimality of $|M_3|$.
 Therefore, $|E_G(C_j, V(G) \setminus C_j) \cap M_3| \in \{0,2\}$. 
 Let $i \in \lbrack n \rbrack$.
 If $E_G(X_i, V(G)\setminus X_i) \cap M_3 \neq \emptyset$, then $E_G(X_i) \cap M_1 = \emptyset$.
 If $E_G(X_i, V(G)\setminus X_i) \cap M_3$ contains exactly one edge $uv$ with $u \in X_i$,
 then the matching $(M \cup \{u_iu\}) \setminus \{uv\}$ is acyclic, 
 which is a contradiction to the minimality of $|M_3|$.
 Therefore, $|E_G(X_i, V(G)\setminus X_i) \cap M_3| \in \{0,2\}$.
 Hence, we obtain that
 \begin{align*}
  \nu_{ac}(G) &= |M_1| + |M_2| + |M_3| \\
	      &\leq n - \frac{|M_3|}{2} + m - \frac{|M_3|}{2} + |M_3| \\
	      &=n+m,
 \end{align*}
 which completes the proof.
\end{proof}
\begin{lemma}
 $\Gamma$ is satisfiable if and only if $\nu_{ac}(G) = \nu_{s}(G)$.
\end{lemma}
\begin{proof}
 Let $\Gamma$ be satisfiable, and let $t: \{x_1,\ldots,x_n \} \rightarrow \{0,1 \}$ be some satisfying truth assignment of $\Gamma$.
 Let $M$  arise from 
 $\{u_if_i : i \in \lbrack n \rbrack \text{ and } t(x_i) = 0\} \cup
 \{u_it_i : i \in \lbrack n \rbrack \text{ and } t(x_i) = 1\}$ by adding, 
 for every $j \in \lbrack m \rbrack$, some edge $v_jw$ 
 where the literal that is identified with the vertex $w$ is true under $t$.
 Suppose that $M$ is not an induced matching, that is, there is some edge $e$ between two
 edges $e_1$ and $e_2$ in $M$. 
 By construction, we may assume that 
 $e_1 = u_iy_i$ for $y_i \in \{f_i , t_i \}$
 and some $i \in \lbrack n \rbrack$, 
 and $e_2 = v_jw$ for some $j \in \lbrack m \rbrack$
 where the literal that is identified with the vertex $w$ is true under $t$.
 Thus, $y_i$ and $w$ are adjacent.
 First, we assume that $y_i = f_i$.
 By construction, the literal
 that is identified with the vertex $w$ is $x_i$,
 which is a contradiction to the choice of $M$. 
 Hence, we may assume that $y_i = t_i$.
 By construction, the literal
 that is identified with the vertex $w$ is $\bar{x}_i$,
 which is a contradiction to the choice of $M$. 
 Since $M$ has size $n+m$, 
 Lemma \ref{l_acy} implies that
 $\nu_{ac}(G) = \nu_{s}(G)$.
 
 Let $\nu_{ac}(G) = \nu_{s}(G)$, which, by Lemma \ref{l_acy}, 
 implies that $\nu_s(G) = n+m$. 
 Let $M$ be some maximum induced matching maximizing 
 $|\{u_i : i \in \lbrack n \rbrack \} \cap V(M)|$.
 If there is some edge $uv$ with $u \in X_i$ and $v \in C_j$ 
 for some $i \in \lbrack n \rbrack$ and $j \in \lbrack m \rbrack$,
 then the matching $M = (M \cup \{ u_iu\}) \setminus \{uv\}$ is induced, 
 which is a contradiction to the maximality of $|\{u_i : i \in \lbrack n \rbrack \} \cap V(M)|$.
 Therefore, $E_G \big( \bigcup_{i=1}^n{X_i}, V(G) \setminus \bigcup_{i=1}^n{X_i} \big) = \emptyset$.
 Moreover, if there is some edge $f_i t_i$ in $M$ for some $i \in \lbrack n \rbrack$, then the matching 
 $(M \cup \{u_if_i\}) \setminus \{f_i t_i\}$ is also induced,
 which is a contradiction to the maximality of  $|\{u_i : i \in \lbrack n \rbrack \} \cap V(M)|$. 
 Since $\nu_s(G) = n+m$, this implies that 
 either $u_if_i$ or $u_it_i$ belong to $M$ for every $i \in \lbrack n \rbrack$, and
 $E_G(C_j) \cap M \neq \emptyset$ for every $j \in \lbrack m \rbrack$.
 Let $t: \{x_1,\ldots,x_n \} \rightarrow \{0,1 \}$ be defined
 as $t(x_i) = 0$ if $u_if_i \in M$ and $t(x_i) = 1$ if $u_i t_i \in M$.
 Suppose, for a contradiction, that $\Gamma$ is not satisfied under $t$, 
 that is, there is some clause $c_j$ such that no literal is true under $t$. 
 By construction, there is some vertex $w \in V(M) \cap C_j$ 
 where its corresponding literal $y$ is not true under $t$.
 By construction, $w$ is adjacent to $f_i$ or $t_i$ for some $i \in [n]$.
 If $w$ is adjacent to $f_i$, then $y = x_i$,
 and, since $M$ is induced, the edge $u_it_i$ is in $M$,
 which implies that $t(x_i) = 1$, a contradiction.
 Hence, we may assume that $w$ is adjacent to $t_i$.
 This implies that $y = \bar{x}_i$ and that the edge $u_if_i$ is in $M$,
 that is, $t(\bar{x}_i) = 1$, a contradiction.   
\end{proof}

\section{Proof of Theorem \ref{t2}}
Given a boolean formula in conjunctive normal form,
{\sc Exact Satisfiability}
is the problem is to 
decide whether there is a truth assignment of the variables so that every
clause contains exactly one true literal.
If there is such a truth assignment, then the instance
is \textit{exact satisfiable}.
 \begin{lemma} \label{l_ur_red}
 {\sc Exact Satisfiability}
 remains NP-complete when restricted to instances where 
 the literals occur only positively, 
 every literal occurs at most three times, 
 every clause has size exactly three,
 and no literal appears twice in one clause.
 \end{lemma}
 \begin{proof} 
 It was proved recently \cite{possw} that {\sc Exact Satisfiability} 
 remains NP-complete when restricted to instances
 where the literals occur only positively, 
 every literal occurs exactly three times, 
 and every clause has size exactly three.
 Let $\Gamma$ be such an instance of {\sc Exact Satisfiability}.
 
 Suppose that some literal $x$ appears
 twice in some clause $c$ of $\Gamma$. 
 This implies that $x$ must be false and that the other literal in $c$ must be true.
 Therefore, some variables already have a unique truth value,
 which might result in a contradiction in which case $\Gamma$ is not exact satisfiable.
 If not, then we delete all variables with a unique truth value.
 If we apply this process iteratively, then
 we obtain an equivalent instance $\Gamma^\prime$
 of {\sc Exact Satisfiability} where
 the literals occur only positively,
 every literal appears exactly three times,
 every clause has size two or three,
 and no literal appears twice in one clause.

 We construct a new instance that is equivalent to $\Gamma^\prime$ where
 the literals occur only positively,
 every literal appears at most three times,
 every clause has size exactly three,
 and no literal appears twice in one clause.
 Suppose that there is a clause of size two with literals $x$ and $y$ in $\Gamma^\prime$.
 Since no literal appears twice in one clause, $x \neq y$.  
 We delete the clause $x \lor y$ and we add four new clauses
 $x \lor y \lor a_1$,
 $a_1 \lor a_2 \lor a_3$,
 $a_1 \lor a_2 \lor a_4$, and
 $a_2 \lor a_3 \lor a_4$.
It is easy to see that this instance of {\sc Exact Satisfiability} satisfies all desired constraints.
Since the only possible solution of the above four clauses is obtained by
assigning $x$ or $y$ to $1$, $a_2$ to $1$, and $a_1$, $a_3$, and $a_4$ to $0$,
the newly constructed instance of {\sc Exact Satisfiability} is equivalent to $\Gamma^\prime$.
\end{proof}
 Let $\Gamma$ be an instance of {\sc Exact Satisfiability} as in Lemma \ref{l_ur_red}
 with variables $x_1,\ldots,x_n$ and clauses $c_1,\ldots,c_m$.
 Let $G$ be the graph that arises from the union of $n$ disjoint copies of a $K_{1,2}$ with vertex sets $X_1,\ldots,X_n$,
 and $m$ disjoint copies of a $K_{1,3}$ with vertex sets $C_1,\ldots,C_m$.
 For every $i \in [n]$, let $f_i$ and $t_i$ be the leaves of $X_i$, and let $u_i$ be the vertex of degree two in $X_i$. 
 For every $j \in [m]$, identify the three leaves of $C_j$ with the literals that belong to the clause $c_j$,
 and let $v_j$ be the vertex of degree three in $C_j$.
 Let $i$ be in $[n]$, 
 let $W$ be the set of vertices in $\bigcup_{j=1}^m C_j$ that are identified with the literal $x_i$,
 and let $J = \{j \in [m] : W \cap C_j \neq \emptyset \}$.
 Furthermore, let $W^\prime = \big( \bigcup_{j \in J}{C_j\setminus \{v_j\}} \big) \setminus W$.
 Now, add the edges $f_iw$ for every $w$ in $W$,
 and the edges $t_iw^\prime$ for every $w^\prime$ in $W^\prime$, see Figure \ref{fig2} for an illustration.
 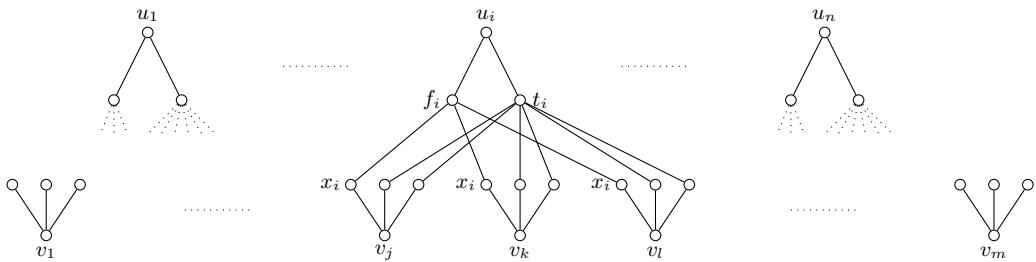
\begin{figure}[H]
\centering\tiny
\begin{tikzpicture}[scale = 0.9] 
	    \foreach \i in {1,2,3} {
	    \node (u\i) at (5*\i,3) {};
	    \node (f\i) at (-0.5+5*\i,2) {};
	    \node (t\i) at (0.5+5*\i,2) {};
	    \draw[-] (u\i) -- (t\i);
	    \draw[-] (f\i) -- (u\i);
	    }
	    \foreach \i in {1,3} {
	    \draw[-,dotted] (f\i) -- (4.3 + 5*\i -5,1.5);
	    \draw[-,dotted] (f\i) -- (4.5 + 5*\i -5,1.5);
	    \draw[-,dotted] (f\i) -- (4.7 + 5*\i -5,1.5);
	    }
	    
	    \foreach \i in {1,2,...,6} {
	    \draw[-,dotted] (t1) -- (5.5-0.7 + 0.2*\i,1.5);
	    \draw[-,dotted] (t3) -- (15.5-0.7 + 0.2*\i,1.5);
	    }
	    \pgftext[x=5cm,y=3.25cm]  {\footnotesize $u_1$};
	    \pgftext[x=10cm,y=3.25cm] {\footnotesize $u_{i}$};
	    \pgftext[x=15cm,y=3.25cm] {\footnotesize $u_{n}$};
	    \pgftext[x=3.5cm,y=-0.25cm] {\footnotesize $v_{1}$};
	    \pgftext[x=8.5cm,y=-0.25cm] {\footnotesize $v_{j}$};
	    \pgftext[x=10.5cm,y=-0.25cm] {\footnotesize $v_{k}$};
	    \pgftext[x=12.5cm,y=-0.25cm] {\footnotesize $v_{l}$};
	    \pgftext[x=17.5cm,y=-0.25cm] {\footnotesize $v_{m}$};
	    \pgftext[x=9.2cm,y=2cm] {\footnotesize $f_{i}$};
	    \pgftext[x=10.8cm,y=2cm] {\footnotesize $t_{i}$};
	    \pgftext[x=7.7cm,y=0.75cm] {\footnotesize $x_i$};
	    \pgftext[x=9.7cm,y=0.75cm] {\footnotesize $x_{i}$};
	    \pgftext[x=11.7cm,y=0.75cm] {\footnotesize $x_{i}$};
	    \draw[-,dotted] (7,2.5) -- (8,2.5);
	    \draw[-,dotted] (12,2.5) -- (13,2.5);
	    
	    \foreach \i in {1,2,3} {
	    \node (a\i) at (6+2*\i,0.75) {};
	    \node (b\i) at (6.5+2*\i,0.75) {};
	    \node (c\i) at (7+2*\i,0.75) {};
	    \node (d\i) at (6.5+2*\i,0) {};
	    \draw[-] (d\i) -- (a\i);
	    \draw[-] (d\i) -- (b\i);
	    \draw[-] (d\i) -- (c\i);
	    }	    
	    \draw[-] (f2) -- (a1);
	    \draw[-] (f2) -- (a2);
	    \draw[-] (f2) -- (a3);
	    \foreach \i in {1,2,3} {
	    \draw[-] (t2) -- (b\i);
	    \draw[-] (t2) -- (c\i);
	    
	    }
	    \draw[-,dotted] (14.5,0.375) -- (15.5,0.375);
	    \draw[-,dotted] (6.5,0.375) -- (5.5,0.375);
	    
	    \foreach \i in {4,5} {
	    \node (a\i) at (-53+14*\i,0.75) {};
	    \node (b\i) at (-52.5 +14*\i,0.75) {};
	    \node (c\i) at (-52 +14*\i,0.75) {};
	    \node (d\i) at (-52.5 +14*\i,0) {};
	    \draw[-] (d\i) -- (a\i);
	    \draw[-] (d\i) -- (b\i);
	    \draw[-] (d\i) -- (c\i);
	    }
\end{tikzpicture}
\caption{The construction for the proof of Theorem \ref{t2}.} \label{fig2}
\end{figure}
Golumbic et al. \cite{gohile} showed that,
if $G$ is a bipartite graph and $M$ is some uniquely restricted matching in $G$,
then $G(M)$ has a vertex of degree $1$ in $G(M)$.
We shall use this within the proof of the following lemma. 
 \begin{lemma} \label{l_ur}
  $\nu_{ur}(G) = n+m$.
 \end{lemma}
 \begin{proof}
  The matching $M$ that arises from $\{u_it_i : i \in [n] \}$ by adding,
  for each $j \in [m]$, an edge between $v_j$ and one of its neighbors is
  uniquely restricted, because all edges in $M$ are pendant edges in $G(M)$.
  Suppose, for a contradiction, that $\nu_{ur}(G) > n+m$,
  and let $M = M_1 \cup M_2 \cup M_3$ be some maximum uniquely restricted matching with 
  \begin{itemize}
   \item $M_1 \subseteq \bigcup_{i=1}^n E_G(X_i)$,
   \item $M_2 \subseteq \bigcup_{j=1}^m E_G(C_j)$, and
   \item $M_3 \subseteq E_G \big( \bigcup_{i=1}^n X_i, V(G) \setminus \bigcup_{i=1}^n X_i \big)$
  \end{itemize}
   minimizing $|M_3|$.
  Since $|M| > n+m$, it follows that $M_3 \neq \emptyset$. 
\begin{claim}\label{c1}
 If $i \in [n]$ is such that
   $M_3 \cap E_G(X_i, V(G) \setminus X_i) \neq \emptyset$,
   then $M_1 \cap E_G(X_i) = \emptyset$.
\end{claim}
\begin{proof}
 Suppose, for a contradiction, that $M_1 \cap E_G(X_i) \neq \emptyset$.
 Let $j \in [m]$ be such that
 $M_3 \cap E_G(X_i,C_j) \neq \emptyset$,
 and let $w_1, w_2$, and $w_3$ be the vertices in $C_j$ distinct from $v_j$
 so that $w_1$ is adjacent to $f_i$, and $w_2$ and $w_3$ are both adjacent to $t_i$.
 
 First, we assume that $f_iw_1$ and $u_it_i$ belong to $M$.
 Since $v_j$ is not covered by $M$,
 the set $M^\prime = \left(M \cup \{v_jw_1\} \right) \setminus \{f_iw_1\}$
 is a matching in $G$, which, by the minimality of $|M_3|$, is not uniquely restricted.
 Since $f_i$ is not covered by $M^\prime$,
 there is an $M^\prime$-alternating cycle in $G$ disjoint from $X_i$, which, by symmetry, can be written
 as $v_jw_1Pw_2v_j$ for some $M^\prime$-alternating path $P$ in $G$.
 Since $P$ is also $M$-alternating, the cycle
 $w_1Pw_2t_iu_if_iw_1$ is $M$-alternating in $G$, which is a contradiction.

 Hence, by symmetry, we may assume that $t_iw_2$ and $u_if_i$ belong to $M$.
 Since $v_j$ is not covered by $M$,
 the set $M^\prime = \left(M \cup \{v_jw_2\} \right) \setminus \{t_iw_2\}$
 is a matching in $G$, which, by the minimality of $|M_3|$, is not uniquely restricted.
 Therefore, there is an $M^\prime$-alternating cycle $C$ that contains the edge $v_jw_2$.
 Since $t_i$ is not covered by $M^\prime$, we have that $X_i \cap V(C) = \emptyset$.
 If $C$ can be written as $w_2v_jw_3Pw_2$ for some $M^\prime$-alternating path $P$, 
 then the cycle $w_2t_iw_3Pw_2$ is $M$-alternating, which is a contradiction.
 Hence, we may assume that $C$ can be written as $w_2v_jw_1Pw_2$
 for some $M^\prime$-alternating path $P$,
 which implies that the cycle
 $w_2t_iu_if_iw_1Pw_2$ is $M$-alternating in $G$, which is a contradiction.
\end{proof}
\begin{claim} \label{c2}
 If $j \in [m]$ is such that
 $M_3 \cap E_G(C_j, V(G) \setminus C_j) \neq \emptyset$,
 then $M_2 \cap E_G(C_j) = \emptyset$.
\end{claim}
\begin{proof}
 Suppose, for a contradiction, that $M_2 \cap E_G(C_j) \neq \emptyset$.
 Let $i \in [n]$ be such that
 $M_3 \cap E_G(X_i,C_j) \neq \emptyset$, and
 let $w_1, w_2$, and $w_3$ be the vertices in $C_j$ distinct from $v_j$ such that
 $w_1$ is adjacent to $f_i$, and $w_2$ and $w_3$ are both adjacent to $t_i$.
 
 First, we assume that $f_iw_1$ and $v_jw_3$ belong to $M$.
 Since $u_i$ is not covered by $M$,
 the set $M^\prime = \left(M \cup \{u_if_i\} \right) \setminus \{f_iw_1 \}$
 is a matching, which, as before, implies that there is an
 $M^\prime$-alternating cycle $C$ that contains the edge $u_if_i$.
 If $v_jw_3$ is not contained in $E(C)$, then the vertices
 $w_1$, $v_j$, and $w_3$ are not contained in $V(C)$,
 which implies that $C$ can be written as $f_iu_it_iPf_i$
 for some $M^\prime$-alternating path $P$ in $G$.
 Since $P$ is also $M$-alternating,
 the cycle $f_iw_1v_jw_3t_iPf_i$ is $M$-alternating, which is a contradiction.
 Hence, we may assume that $v_jw_3$ is contained in $E(C)$.
 In this case, the cycle can either be written as $f_i P w_2v_jw_3P^\prime t_i u_i f_i$
 or as $f_i Q w_3 v_j w_2 Q^\prime t_iu_if_i$.
 In the first case, the cycle $f_iw_1v_jw_3P^\prime t_i w_2 P f_i$ is
 $M$-alternating in $G$, while in the second case
 the cycle $w_2 Q^\prime t_i w_2$ is $M$-alternating,
 which, in both cases, is a contradiction.
 
 Hence, by symmetry, we may assume that $t_iw_2$ and $v_jw_1$ belong to $M$.
 Since $u_i$ is not covered by $M$, the set 
 $M^\prime = \left(M \cup \{u_it_i, v_jw_2 \} \right) \setminus \{t_iw_2, v_jw_1\}$
 is a matching, which, as before, implies that there is an
 $M^\prime$-alternating cycle $C$.
 If $C$ contains exactly one of the edges $u_it_i$ or $v_jw_2$,
 then $C$ can be written as $f_iPt_i u_i f_i$ or
 as $v_jw_3Qw_2v_j$. In the first case,
 the cycle $f_iPt_i w_2v_jw_1f_i$ is $M$-alternating,
 while in the second case
 the cycle $w_3Qw_2t_iw_3$ is $M$-alternating,
 which, in both cases, is a contradiction.
 Hence, we may assume that both $u_it_i$ and $v_jw_2$
 are contained in $E(C)$.
 In this case, the cycle can either be written as $f_i P w_2 v_j w_3 P^\prime t_i u_i f_i$
 or as $f_i Q w_3 v_j w_2 t_i u_i f_i$.
 In the first case, the cycle $w_1 f_i P w_2 t_i P^\prime w_3 v_j w_1$ is
 $M$-alternating in $G$, while in the second case
 the cycle $f_iQw_3v_jw_1f_i$ is $M$-alternating in $G$,
 which, in both cases, is a contradiction.
\end{proof}
By Claim \ref{c1} and \ref{c2}, all edges in $M_1 \cup M_2$ are pendant edges in $G(M)$.
Let $H = G(M_3)$. 
Since $H$ is bipartite, there is a vertex $u$ in $H$ of degree $1$.
Let $uv \in M_3$.
First, we assume that $u \in X_i$ for some $i \in [n]$.
By Claim \ref{c1}, the vertex $u_i$ is not covered by $M$.
Hence, the set $M^\prime = (M \cup \{u_iu\}) \setminus \{uv\}$ is a matching.
Since $u$ is only adjacent to vertices
in $G(M)$ that are covered by edges in $M_2$,
the matching $M^\prime$ is uniquely restricted,
which is a contradiction to the minimality of $|M_3|$.
Hence, we may assume that $u \in C_j$ for some $j \in [m]$.
By Claim \ref{c2}, the vertex $v_j$ is not covered by $M$.
Hence, the set $M^{\prime} = (M \cup \{v_ju\}) \setminus \{uv\}$
is a matching.
Since $u$ is only adjacent to vertices
in $G(M)$ that are covered by edges in $M_1$,
the matching $M^\prime$ is uniquely restricted,
which is a contradiction to the minimality of $|M_3|$.
\end{proof}
\begin{lemma}
 $\Gamma$ is exact satisfiable if and only if $\nu_{ur}(G) = \nu_s(G)$.
\end{lemma}
\begin{proof}
 Let $\Gamma$ be exact satisfiable, and let $t: \{x_1,\ldots,x_n \} \rightarrow \{0,1 \}$ be some satisfying truth assignment of $\Gamma$
 such that each clause contains exactly one literal that is true under $t$.
 Let $M$  arise from 
 $\{ u_if_i : i \in \lbrack n \rbrack \text{ and } t(x_i) = 0\} \cup
 \{ u_i t_i : i \in \lbrack n \rbrack \text{ and } t(x_i) = 1\}$ by adding, 
 for each $j \in \lbrack m \rbrack$, the edge $v_jv$ 
 where $v$ is identified with a literal in $c_j$ that is true under $t$.
 Suppose, for a contradiction, 
 that $M$ is not an induced matching, that is, there is some edge $e$ between two
 edges $e_1$ and $e_2$ in $M$.
 By construction, we may assume that $e_1 = u_iy_i$ for $y_i \in \{f_i , t_i\}$
 and $i \in \lbrack n \rbrack$, and $e_2 = v_jz_j$, where $z_j$ is the vertex in $C_j$
 that is identified with the unique literal in $c_j$ that is true under $t$,
 for some $j \in \lbrack m \rbrack$.
 Therefore, $y_i$ and $z_j$ are adjacent.
 If $y_i = f_i$, then, by construction, 
 $z_j$ is identified with the literal $x_i$,
 which is a contradiction to the choice of $M$.
 Hence, we may assume that $y_i = t_i$,
 which implies that $t(x_i) = 1$.
 Furthermore, by construction, 
 $z_j$ is identified with a literal from the clause $c_j$
 distinct from $x_i$, which, by the choice of $M$,
 is also true under $t$, a contradiction.
 Hence, $M$ is an induced matching of size $n+m$,
 which, by Lemma \ref{l_ur}, implies that $\nu_{ur}(G) = \nu_{s}(G)$.
 
 Let $\nu_{ur}(G) = \nu_s(G)$, which, by Lemma \ref{l_ur}, implies that $\nu_s(G) = n+m$. 
 Let $M = M_1 \cup M_2 \cup M_3$ be some maximum induced matching in $G$ with 
 \begin{itemize}
  \item $M_1 \subseteq \bigcup_{i=1}^n E_G(X_i)$,
  \item $M_2 \subseteq \bigcup_{j=1}^m E_G(C_j)$, and
  \item $M_3 \subseteq E_G \big( \bigcup_{i=1}^n {X_i}, V(G) \setminus \bigcup_{i=1}^n {X_i} \big)$
 \end{itemize}
 minimizing $|M_3|$.
 Suppose, for a contradiction, that $M_3$ is non-empty.
 Let $i \in [n]$ be such that $E_G(X_i, V(G) \setminus X_i) \cap M_3$ is non-empty.
 If $E_G(X_i, V(G) \setminus X_i) \cap M_3 = \{uv\}$ where $u \in X_i$,
 then the matching $(M \cup \{u_iu\}) \setminus \{uv\}$
 is induced, which is a contradiction to the minimality of $|M_3|$.
 Hence, we may assume that $E_G(X_i, V(G) \setminus X_i) \cap M_3 = \{t_iv, f_iw\}$.
 First, we assume that $v,w \in C_j$ for some $j \in [m]$.
 This implies that $v$ and $w$ are the only vertices in $C_j$ that are covered by $M$.
 Hence, the matching $(M \cup \{u_if_i, v_jv\}) \setminus \{t_iv, f_iw\}$
 is induced, which is a contradiction to the minimality of $|M_3|$.
 Hence, we may assume that $v$ belongs to $C_j$ for some $j \in [m]$,
 and that $w$ belongs to $C_{j^\prime}$ for some $j^\prime \in [m] \setminus \{j\}$.
 Again by construction, $v$ is the only vertex in $C_j$ that is covered by $M$.
 Hence, the matching $(M \cup \{u_if_i, v_jv\}) \setminus \{t_iv, f_iw\}$
 is induced, which is a contradiction to the minimality of $|M_3|$.
 Hence, we may assume that $M_3 = \emptyset$.
 Since $|M| = n+m$, this implies
 that $M \cap E_G(X_i)$ and $M \cap E_G(C_j)$ are all non-empty 
 for every $i \in [n]$ and $j \in [m]$.
 Let $t: \{x_1,\ldots,x_n \} \rightarrow \{0,1 \}$ be defined
 as $t(x_i) = 0$ if $u_if_i \in M$ and $t(x_i) = 1$ if $u_it_i \in M$, and 
 suppose, for a contradiction, that $\Gamma$ is not exact satisfied under $t$, 
 that is, there is some clause $c_{\ell} = x_i \lor x_j \lor x_k$ 
 such that not exactly one literal is true under $t$. 
 First, we assume that no literal in $c_{\ell}$ is true under $t$.
 This implies that $u_if_i, u_jf_j$, and $u_kf_k$ belong to $M$,
 which implies that no vertex in $C_{\ell}$ is covered by $M$, a contradiction.
 Hence, we may assume that at least two literals in $c_{\ell}$ are true under $t$,
 which, by symmetry, implies that $u_it_i$ and $u_jt_j$ both belong to $M$.
 Again by construction, this implies that no vertex in $C_{\ell}$ is covered by $M$, a contradiction.
\end{proof}
The graphs constructed in the proof of Theorem \ref{t2} have maximum degree at most $7$.
Replacing $X_1,\ldots,X_n$ by 
$K_{3,3}$'s where the edges of some maximum matching are subdivided once, 
yields the hardness for graphs of maximum degree $5$.
The proof of it proceeds along the lines of the proof of Theorem \ref{t2}.
However, the lemma corresponding to Lemma \ref{l_ur} becomes quite technical,
and so the proof is omitted.
Therefore, in view of \cite{fura2},
for restrictions imposed on the maximum degree, 
the only case left are the graphs with maximum degree $4$.
\section*{Acknowledgement}
I would like to thank Dieter Rautenbach for valuable suggestions 
that helped to improve the presentation of this note.

\end{document}